\documentclass[12pt]{amsart}
\usepackage{amssymb, latexsym, amsthm} 

\newtheorem{thm}{Theorem}[section] 

\newtheorem{cor}[thm]{Corollary}

\newtheorem{lem}[thm]{Lemma}
\newtheorem{prop}[thm]{Proposition}

\newtheorem{rem}[thm]{Remark}

\newcommand\operA[2]{{\if!#2!\operatorname{#1}\else{\operatorname{#1}_{#2}^{\phantom{I}}}\fi}} 

\def\norm{{\operatorname{N}}}

\newcommand{\Trace}[1][]{\if!#1!\operatorname{Tr}\else{\operatorname{Tr}_{#1}^{\phantom{I}}}\fi} 

\newif\ifXY 
\XYtrue     
\ifXY

\input xy
\input xyidioms.tex
\usepackage{xy}
\xyoption{all} %
\fi 

\begin{document}
\title{Pure imaginary roots of quaternion standard polynomials}
\author{Adam Chapman}
\thanks{The author is a Ph.D student under the supervision of Prof. Uzi Vishne, Bar-Ilan University, Ramat-Gan, Israel. email: adam1chapman@yahoo.com}

\begin{abstract}
In this paper, we present a new method for solving standard quaternion equations.
Using this method we reobtain the known formulas for the solution of a quadratic quaternion equation, and provide an explicit solution for the cubic quaternion equation, as long as the equation has at least one pure imaginary root.
We also discuss the number of essential pure imaginary roots of a two-sided quaternion polynomial.
\end{abstract}

\maketitle


\section{Introduction}\label{sec1}

Let $\mathbb{H}=\mathbb{R}+\mathbb{R} i+\mathbb{R} j+\mathbb{R} k$ be the real quaternion algebra, with $i^2=j^2=-1$, $k=i j$ and $j i=-k$.

Every element $z$ in this algebra is therefore of the form $z=c_1+c_2 i+c_3 j+c_4 k$ where $c_1,c_2,c_3,c_4 \in \mathbb{R}$.
Let $\Re(z)=c_1$ and $\Im(z)=z-\Re(z)=c_2 i+c_3 j+c_4 k$. We call $\Re(z)$ the real part of $z$ and $\Im(z)$ the imaginary part.
If $\Re(z)=z$ then $z$ is called pure real and if $\Im(z)=z$ then $z$ is called pure imaginary.
Every element $z$ then can be written as the sum of two elements $r+x$ such that $r=c_1$ is pure real and $x=c_2 i+c_3 j+c_4 k$ is pure imaginary.
By easy calculation one can show that $x^2=-(c_2^2+c_3^2+c_4^2) \in \mathbb{R}$.

The conjugate of $z$ is defined to be $\bar{z}=r-x=c_1-c_2 i-c_3 j-c_4 k$.
The norm of $z$ is defined to be $\norm(z)=z \bar{z}=r^2-x^2=c_1^2+c_2^2+c_3^2+c_4^2 \in \mathbb{R}$.
The norm is known to be a multiplicative function, i.e. $f(z_1 z_2)=f(z_1) f(z_2)$, and for any $c \in \mathbb{R}$, $f(c z)=c^2 f(z)$.

A quaternion polynomial equation with one indeterminate $z$ is called standard if it is of the form $a_n z^n+\dots+a_1 z+a_0=0$ for some $a_0,\dots,a_n \in \mathbb{H}$. Mark that since the quaternion algebra is noncommutative, the order of multiplication is crucial, for instance the equations $a z^2-b=0$, $z a z-b=0$ and $z^2 a-b=0$ are three distinct equations.

In \cite{Janovska} Janovsk\'{a} and Opfer reduced the problem of solving any standard quaternion equation of degree $n$ to a real equation of degree $2 n$.
However, for the case of $n=2$ it is not the optimal, since there are reductions into equations of degree $3$ instead of $4$ (see \cite{Huang and So}, \cite{Au-Yeung}).

Here we present a new method for solving quaternion standard equations.
For the case of $n=2$ it is very similar to the techniques appearing in \cite{Huang and So} and \cite{Au-Yeung}. For the case of $n=3$, if the equation has at least one pure imaginary root, then the problem is reduced to solving a few real equations of degrees no greater than $4$, as opposed to the degree $6$ equation that arises from the method in \cite{Janovska}.

Later in this paper we shall use Wedderburn's decomposition method for standard quaternion polynomials.
The ring of standard (or left) quaternion polynomials $\mathbb{H}[z]$ is simply the ring obtained by adding the variable $z$ to the quaternion algebra with the relations $z a=a z$ for any $a \in \mathbb{H}$. The elements $a z^2$, $z a z$ and $z^2 a$ are the same inside this ring.
However, every polynomial $f(z)$ in that ring has a standard form, where the coefficients lie on the left-hand side of the variable, i.e. $f(z)=a_n z^n+\dots+a_1 z+a_0$ for some $a_0,\dots,a_n \in \mathbb{H}$.
When substituting an element $z_0 \in \mathbb{H}$ in $f(z)$ we substitute in the standard form, i.e. $f(z_0)=a_n z_0^n+\dots+a_1 z_0+a_0$.
We call $a$ a root of $f(z)$ if $f(a)=0$.
Consequently, finding the roots of a polynomial in this ring is equivalent to solving a standard quaternion equation.

It is important to mention that the substitution map $S_{z_0} : \mathbb{H}[z] \rightarrow \mathbb{H}$, taking $S_{z_0}(f(z))=f(z_0)$, is not a ring homomorphism if $z_0$ is not pure real.
For example, if $z_0=i$, $g(z)=z-j$, $h(z)=z+j$ and $f(z)=g(z) h(z)=z^2+1$ then $g(i) h(i)=(i-j) (i+j)=2 k \neq 0$ while $f(i)=0$.

The following statement is known to be true (see \cite{R}):
For given $f(z),g(z),h(z) \in \mathbb{H}[z]$, if $f(z)=g(z) h(z)$ and $a$ is a root of $f(z)$ but not of $h(z)$ then $h(a) a h(a)^{-1}$ is a root of $g(z)$.
Consequently, if $n=\deg(f)$ distinct roots of $f(z)$ are known
then we can factorize $f(z)$ completely to linear factors.
The opposite is not true, i.e. there is no simple algorithm for finding the roots of a polynomial knowing its factorization.

\section{Roots of a quaternion standard polynomial}\label{imsec}

Let there be a monic polynomial $f(z)=z^n+a_{n-1} z^{n-1}+\dots+a_1 z+a_0 \in \mathbb{H}[z]$ where $a_{k-1},\dots,a_0 \in Q$ and $a_0 \neq 0$.

Similarly to the ring of standard polynomials with one variable $\mathbb{H}[z]$, one can look at the ring of polynomials with two variables $\mathbb{H}[r,N]$ where $r a=a r$ and $N a=a N$ for any $a \in \mathbb{H}$.

\begin{lem}\label{twovar}
There exist polynomials $g,h \in \mathbb{H}[r,N]$ such that $f(z_0)=g(r_0,N_0) x_0+h(r_0,N_0)$ for any $z_0 \in \mathbb{H}$, $r_0=\Re(z)$, $x_0=\Im(z)$, $N_0=-x_0^2$.
\end{lem}

\begin{proof}
Let $z_0$ be some arbitrary element in $\mathbb{H}$. $f(z_0)=z_0^n+a_{n-1} z_0^{n-1}+\dots+a_1 z_0+a_0$.
Now, $z_0=r_0+x_0$ for some pure real $r_0$ and some pure imaginary $x_0$.
Since $r_0$ is real, it commutes with $x_0$.
Therefore $z_0^k=\sum_{m=0}^k \binom{k}{m} r_0^{m-k} x_0^m$.
Let $N_0=-x_0^2$. This element is pure real.
For all $1 \leq k \leq n$, $z_0^k=(\sum_{m=0}^{\lfloor \frac{k-1}{2} \rfloor} \binom{k}{2 m+1} (-1)^m N_0^m r_0^{k-(2 m+1)}) x_0+\sum_{m=0}^t \binom{k}{2 m+1} (-1)^m N_0^m r_0^{k-2 m}$.
Let $g_k(r,N)=\sum_{m=0}^{\lfloor \frac{k-1}{2} \rfloor} \binom{k}{2 m+1} (-1)^m N_0^m r_0^{k-(2 m+1)}$ and $h(r,N)=\sum_{m=0}^t \binom{k}{2 m+1} (-1)^m N_0^m r_0^{k-2 m}$.
Now let $g(r,N)=g_n(r,N)+a_{n-1} g_{n-1}(r,N)+\dots+a_1 g_1(r,N)$ and $h(r,N)=h_n(r,N)+a_{n-1} h_{n-1}(r,N)+\dots+a_1 h_1(r,N)+a_0$.
It is easy to see that $f(z_0)=g(r_0,N_0) x_0+h(r_0,N_0)$.
\end{proof}

\begin{thm} \label{normeq}
Given an element $z_0 \in \mathbb{H}$, $x_0,r_0,N_0$ are as in Lemma \ref{twovar}, $z_0$ is a root of $f(z)$ if and only if one of the following conditions is satisfied:
\begin{enumerate}
\item $(r_0,N_0)$ is a solution to both $h(r,N)=0$ and $g(r,N)=0$.
\item $(r_0,N_0)$ is a solution to the equation $-g(r,N) \overline{g(r,N)} g(r,N) N=h(r,N) \overline{g(r,N)} h(r,N)$ and $x_0=-g(r_0,N_0)^{-1} h(r_0,N_0)$.
\end{enumerate}
\end{thm}

\begin{proof}
If $h(r_0,N_0)=g(r_0,N_0)=0$ then $f(z_0)=g(r_0,N_0) x_0+h(r_0,N_0)=0 x_0+0=0$, i.e. $z_0$ is a root of $f(z)$.

If $h(r_0,N_0) \neq 0$ or $g(r_0,N_0) \neq 0$ while $f(z_0)=0$, then $h(r_0,N_0) \neq 0$ and $g(r_0,N_0) \neq 0$, because $g(r_0,N_0) x_0=-h(N_0)$.
Therefore $\overline{g(r_0,N_0)} g(r_0,N_0) x_0=\norm(g(r_0,N_0)) x_0=-\overline{g(r_0,N_0)} h(r_0,N_0)$.

Consequently $-\norm(g(r_0,N_0))^2 N_0=\overline{g(r_0,N_0)} h(r_0,N_0) \overline{g(r_0,N_0)} h(r_0,N_0)$, \\
i.e. $-g(r_0,N_0) \norm(g(r_0,N_0)) N_0=\overline{g(r_0,N_0)} h(r_0,N_0) \overline{g(r_0,N_0)} h(r_0,N_0)$.
This is surely not the trivial equation, because the difference between the lowest degree among the nonzero monomials of the right-hand side of the equation and the lowest degree among the nonzero monomials of the left-hand side of the equation is at least $1$.
Consequently, $(r_0,N_0)$ is a root of the equation $g(r,N) \overline{g(r,N)} g(r,N) N=h(r,N) \overline{g(r,N)} h(r,N)$.
\end{proof}

\section{Solving quadratic equations}

Let $f(z)=z^2+a z+b$. By replacing $z$ with $z-\frac{\Re(a)}{2}$, we may assume that $\Re(a)=0$.
The case of $a=0$ is simple. If $b$ is not pure real then the roots are $\pm \sqrt[4]{\norm{b}} e^{\frac{\theta}{2} \frac{\Im(b)}{\sqrt{\norm(b)}}}$ where $\theta$ is the phase of $b$ in its polar decomposition as a quaternion. If $b$ is pure real then if it is negative then the roots are all the pure imaginary elements whose norms are real square roots of $\norm(b)$. Otherwise, the roots are the real positive and negative square roots of $b$.

Therefore we assume $a \neq 0$. We assumed that $\Re(a)=0$ and therefore $a$ is a nonzero pure imaginary.
Taking $d=\frac{b+a b a^{-1}}{2}$, it is clear that $a d=-d a$ and $a (b-d)=(b-d) a$.
Since $b-d$ commutes with $a$, it is of the form $m+n a$ for some $m,n \in \mathbb{R}$.
The case of $d=0$ is again simple, as in the case of $a=0$.
Consequently we shall assume that $d \neq 0$.

\begin{thm} \label{quadsolve}
Assume $a,d \neq 0$. Let $z_0$ be a root of $f(z)=z^2+a z+m+n a+d$.
If $n \neq 0$ then $r_0=\Re(z_0)$ is a solution to the equation $16 r^6+(-8 a^2+16 m) r^4+(-a^2 (4 m-a^2)+4 a^2 n^2+4 d^2) r^2-a^4 n^2=0$ and $\Im(z_0)=-(2 r_0+a)^{-1} (\frac{1}{2 r_0} a (r_0+n) (2 r_0+a)+d)$.
If $n=0$ then one of the following happens:
\begin{enumerate}
\item $r_0=0$, $N_0=-\Im(z_0)^2$ is a solution to the equation $0=N^2+(a^2-2 m) N+m^2-d^2$ and $\Im(z_0)=-a^{-1} (m+d-N_0)$
\item $r_0$ is a solution to the equation $0=16 r^4+(-8 a^2+16 m) r^2-a^2 (4 m-a^2)+4 d^2$ and $\Im(z_0)=-(2 r_0+a)^{-1} (\frac{1}{2} a (2 r_0+a)+d)$.
\end{enumerate}
\end{thm}

\begin{proof}
The polynomials obtained according to the proof of Lemma \ref{twovar} are in this case $g(r,N)=2 r+a$ and $h(r,N)=r^2-N+a r+b$.
Again $r_0=\Re(z_0)$, $x_0=\Im(z_0)$ and $N_0=-x_0^2$.

Obviously $g(r_0,N_0) \neq 0$, therefore for according to Theorem \ref{normeq}, $(r_0,N_0)$ is a solution to $-g(r,N) \overline{g(r,N)} g(r,N) N=h(r,N) \overline{g(r,N)} h(r,N)$.

We shall solve this equation then.
$-(2 r+a) (2 r-a) (2 r+a) N=(r^2-N+a r+b) (2 r-a) (r^2-N+a r+b)$.

Taking only the part of the equation which anti-commutes with $a$ we obtain

$0=d (2 r-a)(r^2-N+a r+m+n a)+(2 r-a)(r^2-N+a r+m+n a) d=((2 r+a)(r^2-N-a r+m-n a)+(2 r-a)(r^2-N+a r+m+n a))d$

Which means that $0=(2 r+a)(r^2-N-a r+m-n a)+(2 r-a)(r^2-N+a r+m+n a)=4 r^3-4 r N+4 r m-2 a^2 r-2 n a^2$.

If $n \neq 0$ then $r \neq 0$, and so $N=r^2+m-\frac{1}{2} a^2-\frac{1}{2 r} n a^2$.

$h(r,N)=r^2-N+a r+b=r^2-(r^2+m-\frac{1}{2} a^2-\frac{1}{2 r} n a^2)+a r+m+n a+d=\frac{1}{2} a^2+\frac{1}{2 r} n a^2+ a r+n a+d
=\frac{1}{2 r} a (r+n) (2 r+a)+d$

The equation of interest is $-(2 r+a) (2 r-a) (2 r+a) N=h(r,N) (2 r-a) h(r,N)$.
Its part which commutes with $a$ provides us with
$-(2 r+a) (2 r-a) (2 r+a) N=(\frac{1}{2 r} a (r+n) (2 r+a)) (2 r-a) (\frac{1}{2 r} a (r+n) (2 r+a))+d (2 r-a) d
=\frac{1}{4 r^2} (2 r+a) (4 r^2-a^2) a^2 (r+n)^2+(2 r+a) d^2$

Therefore $-(4 r^2-a^2) N=\frac{1}{4 r^2}(4 r^2-a^2) a^2 (r+n)^2+ d^2$, which means that
$0=\frac{1}{4 r^2}(4 r^2-a^2)(4 r^2 (r^2+m-\frac{1}{2} a^2-\frac{1}{2 r} n a^2)+a^2 (r+n)^2)+d^2
=\frac{1}{4 r^2}(4 r^2-a^2) (4 r^4+4 r^2 m-2 a^2 r^2-2 r n a^2+a^2 r^2+2 a^2 r n+a^2 n^2)+d^2
=\frac{1}{4 r^2}(4 r^2-a^2) (4 r^4+4 r^2 m-a^2 r^2+a^2 n^2)+d^2$.

Consequently $16 r^6+(-8 a^2+16 m) r^4+(-a^2 (4 m-a^2)+4 a^2 n^2+4 d^2) r^2-a^4 n^2=0$.

If $n=0$ then $0=4 r^3-4 r N+4 r m-2 a^2 r=r (4 r^2-4 N+4 m-2 a^2)$, which means that either $r=0$ or $N=r^2+m-\frac{1}{2} a^2$.
If $r=0$ then $a^3 N=(-N+m+d) (-a) (-N+m+d)$. Taking only the part which commutes with $a$ we obtain
$a^3 N=-(-N+m)^2 a+a d^2$, hence $a^2 N=-N^2+2 m N-m^2+d^2$, and consequently $0=N^2+(a^2-2 m) N+m^2-d^2$.

If $N=r^2+m-\frac{1}{2} a^2$ then $h(r,N)=r^2-N+a r+b=r^2-(r^2+m-\frac{1}{2} a^2)+a r+m+d=\frac{1}{2} a^2+a r+d=\frac{1}{2} a (2 r+a)+d$.
From $-(2 r+a) (2 r-a) (2 r+a) N=h(r,N) (2 r-a) h(r,N)$ we obtain $-(2 r+a) (2 r-a) (2 r+a) N=(\frac{1}{2} a (2 r+a)+d) (2 r-a) (\frac{1}{2} a (2 r+a)+d)$. Taking the part which commutes with $a$ we get
$-(2 r+a) (2 r-a) (2 r+a) N=\frac{1}{4} a^2 (2 r+a)^2 (2 r-a)+(2 r+a) d^2$.
Therefore $-(4 r^2-a^2) N=\frac{1}{4} a^2 (4 r^2-a^2)+ d^2$, hence
$0=\frac{1}{4} (4 r^2-a^2) (4 (r^2+m-\frac{1}{2} a^2)+a^2)+d^2=\frac{1}{4} (4 r^2-a^2) (4 r^2+ 4m-a^2)+d^2$
and consequently $0=16 r^4+(-8 a^2+16 m) r^2-a^2 (4 m-a^2)+4 d^2$.
\end{proof}

\section{Pure imaginary roots of a quaternion standard polynomial}

Let $f(z)$, $g(r,N)$ and $h(r,N)$ as in Lemma \ref{twovar}. Let $g(N)=g(0,N)$ and $h(N)=h(0,N)$. For every pure imaginary $z_0$, $f(z_0)=g(N_0) z_0+h(N_0)$ where $N_0=\norm(z_0)=-z_0^2$.
In particular, $\deg(g)=\lfloor \frac{\deg{f}-1}{2} \rfloor$ and $\deg{h} \leq \lfloor \frac{\deg{f}}{2} \rfloor$.

The following corollary is an easy result of Theorem \ref{normeq}:
\begin{cor}\label{pureim}
A pure imaginary element $z_0$ of norm $N_0$ is a root of $f(z)$ if and only if one of the following conditions is satisfied:
\begin{enumerate}
\item $N_0$ is a solution to both $h(N)=0$ and $g(N)=0$.
\item $N_0$ is a solution to the equation $-g(N) \overline{g(N)} g(N) N=h(N) \overline{g(N)} h(N)$ and $z_0=-g(N_0)^{-1} h(N_0)$.
\end{enumerate}
\end{cor}

\begin{prop}\label{infcor}
The polynomial $f(z)$ has infinitely many pure imaginary roots if and only if $h(N)=0$ and $g(N)=0$ have a common real solution.
\end{prop}

\begin{proof}
If $h(N)=0$ and $g(N)=0$ have a common real solution $N_0$ then every element $z_0 \in Q$ satisfying $-z_0^2=N_0$ is a root of $f(z)$.

If $h(N)$ and $g(N)$ have no common root, and $z_0$ is a pure imaginary root of $f(z)$ of norm $N_0$, then $h(N_0) \neq 0$ and $g(N_0) \neq 0$.
On the other hand, $N_0$ is a solution to the equation $g(N) \overline{g(N)} g(N) N=h(N) \overline{g(N)} h(N)$.
The degree of the left-hand side of this equation is $3 \deg(g)+1$, while the degree of the right-hand side is $2 \deg(h)+\deg(g)$. There is an equality only if $2 \deg(g)+1=2 \deg(h)$, but that can never happen, therefore the equation is not trivial, which means that by splitting the equation into four (according to the structure of $\mathbb{H}$ as a vector space over $\mathbb{R}$, i.e. $\mathbb{R}+\mathbb{R} i+\mathbb{R} j+\mathbb{R} k$) we have at least one nontrivial equation.
Consequently, the number of roots of this system is finite, and therefore the number of pure imaginary roots of $f(z)$ is finite.
\end{proof}

\begin{rem}
If $z_0$ is a pure imaginary root then $\norm(g(N_0)) z_0=-\overline{g(N_0)} h(N_0)$.
Since $\Re(z_0)=0$, we obtain $0=\Re(-\overline{g(N_0)} h(N_0))$.
If this equation is not trivial, then it has a finite set of roots which contains all the pure imaginary roots of the original equation.
\end{rem}

\section{Solving cubic quaternion equations with at least one pure imaginary root}

\begin{lem} \label{aplem}
For any polynomial $p(z) \in \mathbb{H}[z]$,
if $z_0 \neq a$ is a root of $f(z)=p(z)(z-a)$ then $0=z_0^2-(a+b) z_0+b a$ for some root $b$ of $p(z)$
\end{lem}

\begin{proof}
According to Wedderburn's method, $b=(z_0-a) z_0 (z_0-a)^{-1}$ is a root of $p(z)$.
Hence $b (z_0-a)=(z_0-a) z$, i.e. $0=z_0^2-(a+b) z_0+b a$.
\end{proof}

\begin{rem}
If the decomposition into linear factors of a given polynomial $f(z) \in \mathbb{H}[z]$ is known, then the question of finding its roots becomes (inductively) a sequence of quadratic equations one has to solve. Over the quaternion algebra the quadratic equations are solvable and so one can obtain the roots of any standard polynomial if he knows its decomposition into linear factors.
\end{rem}

Let $f(z)$ be a quaternion standard cubic polynomial.
The equation $-g(N) \overline{g(N)} g(N) N=h(N) \overline{g(N)} h(N)$ (from Corollary \ref{pureim}) is with one variable $N$ and is of degree $4$ at most, and therefore its real roots can be expressed in terms of radicals, which means that the pure imaginary roots of $f(z)$ can also be expressed in those terms.

Assume $f(z)$
has one such root $a$, then $f(z)=p(z) (z-a)$.
The polynomial $p(z)$
is quadratic and therefore its roots can be formulated.
Consequently, $p(z)$ can be fully factorized into linear factors and so is $f(z)$.
Furthermore, according to Lemma \ref{aplem}, the roots of $f(z)$ are at hand.

\subsection{Example}
Consider the polynomial $f(z)=z^3+(2+i j) z+i-j \in \mathbb{H}[z]$.

$g(N)=-N+2+i j$ and $h(N)=i-j$.
They have no common root, so we turn to solve
$-g(N) \norm(g(N)) N=h(N) \overline{g(N)} h(N)$
i.e.
$-(-N+2+i j) ((-N+2)^2+1) N=(i-j) (-N+2-i j) (i-j)$

$-(-N+2+i j) (N^2-4 N+5)  N=(-N+2+i j) (i-j) (i-j)$

$-(N^2-4 N+5)  N=-2$

$N^3-4 N^2+5 N-2=0$

In general this equation could be split into up to four equations according to the basis of $\mathbb{H}$ as an $\mathbb{R}$-vector space.
However, in this case, $N^3-4 N^2+5 N-2$ is pure real and has no imaginary part, which means that we have to solve only one cubic real equation.

Therefore either $N=1$ or $N=2$.
According to Theorem \ref{normeq}, the corresponding roots are $-g(N)^{-1} h(N)$, i.e
$z_1=-\frac{1}{2}(1-i j)(i-j)=-\frac{1}{2}(i-j-j-i)=j$
for $N=2$ we have $z_2=-(i j)(i-j)=-i-j$.

Consequently $f(z)=p(z) (z-j)$. Next goal is to calculate $p(z)$.

\begin{rem} \label{factrem}
Let us recall how $f(z)$ is decomposed into $p(z) (z-a)$ given a root $a$:

$f(z)=z^n+c_{n-1} z^{n-1}+\dots+c_0$

$f(a)=a^n+c_{n-1} a^{n-1}+\dots+c_0$

$f(z)=f(z)-0=f(z)-f(a)=(z^n-a^n)+c_{n-1} (z^{n-1}-a^{n-1})+\dots+c_1 (z-a)
=((z^{n-1}+a z^{n-2}+\dots+a^{n-1})+c_{n-1} (z^{n-2}+\dots+a^{n-2})+\dots+c_1)(z-a)$

$p(z)=(z^{n-1}+a z^{n-2}+\dots+a^{n-1})+c_{n-1} (z^{n-2}+\dots+a^{n-2})+\dots+c_1$.
\end{rem}

Consequently $p(z)=(z^2+j z-1)+2+i j=z^2+j z+1+i j$.

$-i-j$ is a root of $f(z)$ but not of $z-j$, hence according to Remark \ref{factrem}, $(-i-2 j) (-i-j) (-i-2 j)^{-1}=\frac{1}{5}(-i-2 j) (-i-j) (i+2 j)=\frac{1}{5} (-1-2 i j+i j-2)(i+2 j)=\frac{1}{5} (-3-i j) (-i-j)=\frac{1}{5} (3 i+3 j+j-i)=\frac{1}{5}(2 i+4 j)$ is a root of $p(z)$.

The second and final root of $p(z)$ (which can be obtained using the methods) is $i$.

Again, due to Wedderburn, $p(z)=(z+i+1+i j)(z-i)$, which means that $f(z)=(z+1+i+i j)(z-i)(z-j)$

Let $z_0$ be some root of $f(z)$.
According to Lemma \ref{aplem}, since $i$ is a root of $p(z)$ and is different from $\frac{1}{5}(2 i+4 j)$, $z_0$ must correspond to it, which means that $z_0^2-(j+i) z_0+i j=0$.
$j$ is a root, however it is already known to be a root of $f(z)$ so we look for the other one.
Let $t=z_0-j$ and so $t^2-i t+t j=0$.
Let $r=t^{-1}$ and so $1-r i+j r=0$.
$r=c_1+c_i i+c_j j+c_{i j} i j$, so we obtain the following linear system
\begin{eqnarray}
1+c_i-c_j&=&0\\
-c_1+c_{i j}&=&0\\
-c_{i j}+c_1&=&0\\
c_j-c_i=0
\end{eqnarray}
This system has no solution. Therefore, $f(z)$ has no roots besides $j$ and $-i-j$.

\section{A note on quadratic two-sided polynomials}

A two-sided polynomial is a polynomial of the form $f(z)=z^n+a_{n-1} z^{n-1} b_{n-1}+\dots+a_1 z b_1+c$.
Unlike the polynomials in $\mathbb{H}[z]$, when substituting an element $z_0 \in \mathbb{H}$ in the two-sided polynomial we follow the two-sided form instead of moving all the coefficients to the left, i.e. $f(z_0)=z_0^n+a_{n-1} z_0^{n-1} b_{n-1}+\dots+a_1 z_0 b_1+c$.
In \cite{Janovska2} Janovsk\'{a} and Opfer provided an example of a quadratic two-sided polynomial with more than two roots with pairwise distinct norms. (These are called essential roots in that paper.)

This is apparently impossible with pure imaginary roots, as the following proposition suggests:

\begin{prop}
The number of pure imaginary roots of $f(z)=z^2+a z b+c$, assuming $a,b,c \neq 0$, with pairwise distinct norms, is at most two.
\end{prop}

\begin{proof}
Let $z_0$ be a pure imaginary root of norm $N_0$.
Therefore $-N_0+a z_0 b+c=0$, i.e. $N_0-c=a z_0 b$,
hence $a^{-1} b^{-1} N_0-a^{-1} c b^{-1}=z_0$, which means that $a^{-1} b^{-1} a^{-1} b^{-1} N_0^2-(a^{-1} b^{-1} a^{-1} c b^{-1}+a^{-1} c b^{-1} a^{-1} b^{-1}) N_0+a^{-1} c b^{-1} a^{-1} c b^{-1}=-N_0$.
Consequently, $N_0$ is a root of the non-trivial polynomial $p(N)=a^{-1} b^{-1} a^{-1} b^{-1} N^2+(1-a^{-1} b^{-1} a^{-1} c b^{-1}-a^{-1} c b^{-1} a^{-1} b^{-1}) N+a^{-1} c b^{-1} a^{-1} c b^{-1}$.
Hence, the number of pure imaginary roots of $f(z)$ with pairwise distinct norms does not exceed $2$.
\end{proof}

\newcommand\paper[7]{{{#1},\ {\it{#2}},\ {#3}\ {\bf{#4}}\if!#5!\else(#5)\fi,\ {#6},\ ({#7}).}} 
\newcommand\book[4]{{{#1},\ {``{#2}''}{\if!#3!\relax\else{,\ {#3}}\fi}{\if!#4!\relax\else{,\ {#4}}\fi}.}} 
\newcommand\thesis[4]{{{#1},\ {{#2}}, Doctoral Dissertation, {#3},\ {#4}.}} 
\newcommand\mscthesis[4]{{{#1},\ {{#2}}, MSc. thesis, {#3},\ {#4}.}} 
\newcommand\preprint[4]{{{#1},\ {\it{#2}}, preprint, {#3}\ ({#4}).}} 
\newcommand\paperinBook[5]{{{#1},\ {\it{#2}},\ in {#3}, {#4}, {#5}.}} 
\newcommand\paperinEdBook[6]{{{#1},\ {\it{#2}},\ in {#3}, {#4} {#5}, {#6}.}} 
\newcommand\paperTrans[9]{{{#1},\ {\it{#2}},\ {#3}, {#4}, {#5}. Trans. {#6}, {#7}, {#8}, ({#9}).}} 
\newcommand\paperCorr[8]{{{#1},\ {\it{#2}},\ {#3},\ {#4},\ {#5},\ corr. ibid. {#6}, {#7}, {(#8)}.}} 
\newcommand\submitted[4]{{{#1},\ {\it{#2}}, submitted, {#3}\ ({#4}).}} 

\end{document}